\documentclass[12pt,draft]{amsart}
\usepackage{amssymb,mathabx,times}

\makeatletter
\def\@strippedMR{}
\def\@scanforMR#1#2#3\endscan{
  \ifx#1M\ifx#2R\def\@strippedMR{#3}
  \else\def\@strippedMR{#1#2#3}
  \fi\fi}
\renewcommand\MR[1]{\relax\ifhmode\unskip\spacefactor3000 \space\fi
  \@scanforMR#1\endscan
  MR\MRhref{\@strippedMR}{\@strippedMR}}
\makeatother

\parskip=\medskipamount

\newtheorem*{Thm*}{Theorem}
\newtheorem{Thm}{Theorem}
\newtheorem{Cor}[Thm]{Corollary}
\newtheorem{Prop}[Thm]{Proposition}
\newtheorem{Lemma}{Lemma}

\theoremstyle{definition}
\newtheorem{Defn}{Definition}
\newtheorem{Notation}[Defn]{Notation}

\newtheorem{Remark}{Remark}
\newtheorem{Example}{Example}

\newtheorem*{Question}{Question}

\newcommand{\mf}[1]{\mathbb{#1}}
\newcommand{\mc}[1]{\mathcal{#1}}
\newcommand{\mb}[1]{\mathbf{#1}}

\DeclareMathOperator{\Var}{\mathrm{Var}}

\newcommand{\abs}[1]{\left\vert#1\right\vert}

\newcommand{\set}[1]{\left\{#1\right\}}
\newcommand{\ip}[2]{\left \langle #1, #2 \right \rangle}
\newcommand{\Span}[1]{\mathrm{Span} \left( #1 \right)}
\newcommand{\eps}{\varepsilon}

\newcommand{\Dist}{\mathcal{D}_{alg}}

\allowdisplaybreaks[1]

\title{Free evolution on algebras with two states II}
\author{Michael Anshelevich}
\thanks{This work was supported in part by NSF grants DMS-0900935 and DMS-1160849.}
\address{Department of Mathematics, Texas A\&M University, College Station, TX 77843-3368}
\email{manshel@math.tamu.edu}
\date{\today}

\begin{document}

\begin{abstract}
Denote by $\mc{J}$ the operator of coefficient stripping. We show that for any free convolution semigroup $\set{\mu_t :t \geq 0}$ with finite variance, applying a single stripping produces semicircular evolution with non-zero initial condition, $\mc{J}[\mu_t] = \rho \boxplus \sigma_{\beta, \gamma}^{\boxplus t}$, where $\sigma_{\beta, \gamma}$ is the semicircular distribution with mean $\beta$ and variance $\gamma$. For more general freely infinitely divisible distributions $\tau$, expressions of the form $\widetilde{\rho} \boxplus \tau^{\boxplus t}$ arise from stripping $\widetilde{\mu}_t$, where $\set{(\widetilde{\mu}_t, \mu_t) : t \geq 0}$ form a semigroup under the operation of two-state free convolution. The converse to this statement holds in the algebraic setting. Numerous examples illustrating these constructions are computed. Additional results include the formula for generators of such semigroups.
\end{abstract}


\maketitle

\section{Introduction}

A probability measure $\mu$ on $\mf{R}$ all of whose moments are finite can be described by two sequences of Jacobi parameters:
\[
J(\mu) =
\begin{pmatrix}
\beta_0, & \beta_1, & \beta_2, & \beta_3, & \ldots \\
\gamma_0, & \gamma_1, & \gamma_2, & \gamma_3, & \ldots
\end{pmatrix}.
\]
For example, its Cauchy transform
\[
G_\mu(z) = \int_{\mf{R}} \frac{1}{z - x} \,d\mu(x)
\]
(which determines the measure) has the continued fraction expansion
\[
G_\mu(z) =
\cfrac{1}{z - \beta_0 -
\cfrac{\gamma_0}{z - \beta_1 -
\cfrac{\gamma_1}{z - \beta_2 -
\cfrac{\gamma_2}{z - \beta_3 -
\cfrac{\gamma_3}{z - \ldots}}}}}
\]
Define new measures $\Phi[\mu]$ and $\mc{J}[\mu]$ by the right and left shifts on Jacobi parameters:
\[
J(\Phi[\mu]) =
\begin{pmatrix}
0, & \beta_0, & \beta_1, & \beta_2, & \ldots \\
1, & \gamma_0, & \gamma_1, & \gamma_2, & \ldots
\end{pmatrix}
\]
and
\[
J(\mc{J}[\mu]) =
\begin{pmatrix}
\beta_1, & \beta_2, & \beta_3, & \beta_4, & \ldots \\
\gamma_1, & \gamma_2, & \gamma_3, & \gamma_4, & \ldots
\end{pmatrix}.
\]
$\mc{J}$ is sometimes called coefficient stripping. Actually, both $\Phi$ and $\mc{J}$ can be defined more generally: $\Phi$ for any probability measure, and $\mc{J}$ for any probability measure with finite variance. See Definition~\ref{Defn:Phi-J}.

Denote
\[
d \sigma_{\beta, \gamma}(x) = \frac{1}{2 \pi \gamma} \sqrt{4\gamma - (x- \beta)^2} \,dx
\]
the semicircular distribution with mean $\beta$ and variance $\gamma$, $\sigma = \sigma_{0,1}$ the standard semicircular distribution, and $\boxplus$ the operation of free convolution. The semicircular family $\set{\sigma_{\beta t, \gamma t} = \sigma_{\beta, \gamma}^{\boxplus t} : t \geq 0}$ forms a free convolution semigroup. General free convolution semigroups
\[
\set{\mu_t : t \geq 0}
\]
with mean zero and variance $t$ are indexed by probability measures $\rho$. In Proposition~9 of \cite{Ans-Generator}, we showed that for any such free convolution semigroup,
\begin{equation*}
\mc{J}[\mu_t] = \rho \boxplus \sigma^{\boxplus t},
\end{equation*}
so that the ``once-stripped'' free convolution semigroup is always a ``free heat evolution'' started at $\rho$. Needless to say, this statement has no analog for semigroups with respect to usual convolution. In the first result of the paper, we extend this formula to the case of general finite variance: for a free convolution semigroup $\set{\mu_t}$ with mean $\beta t$ and non-zero variance $\gamma t$,
\begin{equation}
\label{Eq:Free-evolution-intro}
\mc{J}[\mu_t] = \rho \boxplus \sigma_{\beta, \gamma}^{\boxplus t}.
\end{equation}

Since any free convolution semigroup, when stripped, always gives a semicircular evolution, it is natural to ask for which families of measures $\set{\widetilde{\mu}_t : t \geq 0}$ is
\begin{equation}
\label{Eq:Two-state-evolution-short}
\mc{J}[\widetilde{\mu}_t] = \widetilde{\rho} \boxplus \tau^{\boxplus t}
\end{equation}
for other measures $\tau$. The main result of the article is that if this is the case, there exists a free convolution semigroup $\set{\mu_t : t \geq 0}$ such that the family of pairs of measures $\set{(\widetilde{\mu}_t, \mu_t): t \geq 0}$ forms a semigroup under the operation $\boxplus_c$ of \emph{two-state free convolution}. Note that formula~\eqref{Eq:Two-state-evolution-short} can sometimes be assigned a meaning even if $\tau$ is not freely infinitely divisible. For example, if $\widetilde{\rho} = \nu \boxplus \tau$ for some $\nu$, then for general probability measures $\tau, \nu$ there exists a family of measures forming the first component of the two-state free convolution semigroup such that
\[
\mc{J}[\widetilde{\mu}_t] = \nu \boxplus \tau^{\boxplus (1 + t)}
\]
(recall that in free probability, $\tau^{\boxplus (1 + t)}$ is well defined for any $\tau$ as long as $t \geq 0$). The most general case covered by the main theorem of the article (Theorem~\ref{Thm:Two-state-free-evolution}) is that for some semigroups,
\[
\mc{J}[\widetilde{\mu}_t] = \widetilde{\rho} \boxplus \omega^{\boxplus (t/p)},
\]
where $\tau = \omega^{\boxplus (1/p)}$ need not even be a positive measure, but where the \emph{subordination distribution} $\omega \boxright \widetilde{\rho}$ is freely infinitely divisible. It is unclear at this point whether every two-state free convolution semigroup (with finite variance) is of this form. Nevertheless, a large group of examples fit into this framework: free convolution semigroups, Boolean convolution semigroups, two-state free Brownian motions, and two-state free Meixner distributions. Moreover, in the last section of the paper we show that in the algebraic setting, when $(\widetilde{\mu}_t, \mu_t)$ are linear functionals on polynomials but do not necessarily come from positive measures, formula~\eqref{Eq:Two-state-evolution-short} does always hold for some (not necessarily positive) $\tau$. In that section we also prove a basic formula for the moment generating function of the multivariate subordination distribution (see below), which really belongs on the long list of properties of that distribution proven in \cite{Nica-Subordination}.

The other aspects of two-state free convolution semigroups are investigated at the end of Section~\ref{Section:Complex}. We compute the two-state version of Voiculescu's evolution equation for the Cauchy transform. Then we combine it with the preceding results to find the formula for the generators of two-state free convolution semigroups with finite variance.

Finally, we would like to explain the connection between this article and part I of the same title \cite{AnsEvolution}. In \cite{Belinschi-Nica-B_t,Belinschi-Nica-Free-BM}, Belinschi and Nica proved that the eponymous family of transformations $\set{\mf{B}_t: t \geq 0}$, is related to the free heat evolution via
\begin{equation}
\label{Eq:Phi-evolution}
\mf{B}_t[\Phi[\rho]] = \Phi[\rho \boxplus \sigma^{\boxplus t}].
\end{equation}
Equation~\eqref{Eq:Free-evolution-intro} follows from this observation after only a small amount of work. In part I, we constructed a two-variable map $\Phi[\cdot, \cdot]$ and proved that
\begin{equation}
\label{Eq:Phi-2-evolution}
\mf{B}_t[\Phi[\tau, \widetilde{\rho}]] = \Phi[\tau, \widetilde{\rho} \boxplus \tau^{\boxplus t}].
\end{equation}
Moreover, the transformation $\Phi[\cdot, \cdot]$ as defined in \cite{AnsEvolution} also comes from two-state free probability theory. In \cite{Nica-Subordination}, Nica observed that $\Phi[\tau, \widetilde{\rho}]$ is closely related to the subordination distribution $\tau \boxright \widetilde{\rho}$, which is a more important object in free probability, and so will be used in computations in this paper.

At this point the evolution formula~\eqref{Eq:Phi-2-evolution} is only proven for measures all of whose moments are finite, while we are interested in a more general class of measures with finite variance. Moreover, the derivation of \eqref{Eq:Free-evolution-intro} from \eqref{Eq:Phi-evolution} does not generalize to a derivation of \eqref{Eq:Two-state-evolution-short} from \eqref{Eq:Phi-2-evolution}; the proof of \eqref{Eq:Two-state-evolution-short} is quite different. Nevertheless, both this article and part I involve two-state free probability theory, and generalization of semicircular evolution to more general free convolution semigroups.

\textbf{Acknowledgments.} The author is grateful to Dan Voiculescu for asking the question which led to Proposition~\ref{Prop:PDE}, to Hari Bercovici, Serban Belinschi, and Wojtek M{\l}otkowski for discussions leading to Lemma~\ref{Lemma:Hari} and Example~\ref{Example:Counterexample}, and to the referee for numerous helpful comments.

\section{Background}
\label{Section:Preliminaries}

\begin{Notation}
Denote by $m[\mu]$ and $\Var[\mu]$ the mean and variance of $\mu$,
\[
\mc{P} = \set{\text{probability measures on } \mf{R}},
\]
\[
\mc{P}_2 = \set{\mu \in \mc{P} : \Var[\mu] < \infty},
\]
\[
\mc{P}_{0,1} = \set{\mu \in \mc{P}_2 : m[\mu] = 0, \Var[\mu] = 1},
\]
\[
\mc{ID}^{\boxplus} = \set{\mu \in \mc{P} : \mu \text{ is } \boxplus\text{-infinitely divisible}}.
\]
For a probability measure $\mu$ on $\mf{R}$, its Cauchy transform is
\[
G_\mu(z) = \int_{\mf{R}} \frac{1}{z - x} \,d\mu(x),
\]
and its $F$-transform is
\[
F_\mu(z) = \frac{1}{G_\mu(z)}
\]
(for a function $f$, $f^{-1}$ will denote its compositional rather than a multiplicative inverse).
\end{Notation}

\subsection{Convolutions}

For $\mu \in \mc{P}$, define its Voiculescu transform $\phi_\mu$ by
\[
(\phi_\mu \circ F_\mu)(z) + F_\mu(z) = z
\]
See \cite{BV93,VDN}. The free convolution of two measures $\mu \boxplus \nu$ is determined by the equality
\[
\phi_{\mu \boxplus \nu} = \phi_\mu + \phi_\nu
\]
on a domain. A free convolution semigroup is a weakly continuous family $\set{\mu_t: t \geq 0} \subset \mc{P}$ satisfying
\[
\mu_t \boxplus \mu_s = \mu_{t+s}.
\]
In this case we denote $\mu_t = \mu^{\boxplus t}$. A measure $\mu$ is $\boxplus$-infinitely divisible if $\mu = \mu_1$ for some free convolution semigroup. A fundamental result in \cite{Nica-Speicher-Multiplication}, extended to measures with unbounded support in \cite{Belinschi-Bercovici-Partially-defined}, is that for any $\mu \in \mc{P}$, $\mu^{\boxplus t}$ is defined for $t \geq 1$.

We will refer to the set
\[
\set{(\beta, \gamma, \rho) : \beta \in \mf{R}, \gamma > 0, \rho \in \mc{P}} \cup \set{(\beta, 0, \cdot) : \beta \in \mf{R}}
\]
as \emph{canonical triples}. By a result of Maassen \cite{Maa92}, $\boxplus$-convolution semigroups with finite variance
\[
\set{\mu_t : t \geq 0, \Var[\mu_1] < \infty}
\]
are in bijection with canonical triples, the bijection being given by
\begin{equation}
\label{Eq:Maassen-rep}
\phi_{\mu_t}(z) = \beta t + \gamma t G_\rho(z).
\end{equation}
Here $\beta = m[\mu_1]$ and $\gamma = \Var[\mu_1]$. $\boxplus$-convolution semigroups with zero variance are of the form $\mu_t = \delta_{\beta t}$, and so correspond to $(\beta, 0, \cdot)$ with $\gamma = 0$ and $\rho$ undefined.

Similarly, for $\widetilde{\mu}, \mu \in \mc{P}$, define the two-state Voiculescu transform $\phi_{\widetilde{\mu}, \mu}$ by
\begin{equation}
\label{Eq:Two-state-R-transform}
(\phi_{\widetilde{\mu}, \mu} \circ F_\mu)(z) + F_{\widetilde{\mu}}(z) = z.
\end{equation}
See \cite{Krystek-Conditional,Wang-Additive-c-free}. The two-state free convolution of two pairs of measures
\[
(\rho, \mu \boxplus \nu) = (\widetilde{\mu},\mu) \boxplus_c (\widetilde{\nu}, \nu)
\]
is determined by the equality
\[
\phi_{\rho, \mu \boxplus \nu} = \phi_{\widetilde{\mu},\mu} + \phi_{\widetilde{\nu}, \nu}
\]
on a domain. A two-state free convolution semigroup is a component-wise weakly continuous family $\set{(\widetilde{\mu}_t, \mu_t): t \geq 0}$ satisfying
\[
(\widetilde{\mu}_t, \mu_t) \boxplus_c (\widetilde{\mu}_s, \mu_s) = (\widetilde{\mu}_{t+s}, \mu_{t+s}).
\]
In this case we denote $(\widetilde{\mu}_t, \mu_t) = (\widetilde{\mu}, \mu)^{\boxplus_c t}$. The pair $(\widetilde{\mu}, \mu)$ is $\boxplus_c$-infinitely divisible if $(\widetilde{\mu}, \mu) = (\widetilde{\mu}_1, \mu_1)$ for some two-state free convolution semigroup.

For a fixed free convolution semigroup $\set{\mu_t : t \geq 0}$, the $\boxplus_c$-convolution semigroups $\set{(\widetilde{\mu}_t, \mu_t)}$ such that $\widetilde{\mu}_1$ has finite variance are in bijection with (relative) canonical triples $(\widetilde{\beta}, \widetilde{\gamma}, \widetilde{\rho})$, the bijection being given by
\begin{equation}
\label{Eq:Two-state-Maassen}
\phi_{\widetilde{\mu}_t, \mu_t}(z) = \widetilde{\beta} t + \widetilde{\gamma} t G_{\widetilde{\rho}}(z).
\end{equation}
Here $\widetilde{\beta} = m[\widetilde{\mu}_1]$ and $\widetilde{\gamma} = \Var[\widetilde{\mu}_1]$. This does not appear to be stated explicitly, but follows from the description of general two-state freely infinitely divisible distributions in Theorem~4.1 of \cite{Wang-Additive-c-free}. Again, the case $\Var[\widetilde{\mu}_1] = 0$ can be included by setting $\widetilde{\gamma} = 0$ and leaving $\widetilde{\rho}$ undefined.

The Boolean convolution $\mu \uplus \nu$ is defined by
\[
(\mu, \delta_0) \boxplus_c (\nu, \delta_0) = (\mu \uplus \nu, \delta_0).
\]
More explicitly, $\phi_{\mu, \delta_0}(z) = z - F_\mu(z)$, so
\[
z - F_{\mu \uplus \nu}(z) = (z - F_\mu(z)) + (z - F_\nu(z)).
\]
Any distribution is $\uplus$-infinitely divisible, so $\mu^{\uplus t}$ is always defined for any $t \geq 0$.

Finally, a few arguments in the article simplify with the use of the monotone convolution $\mu \rhd \nu$, defined by
\[
F_{\mu \rhd \nu} = F_\mu \rhd F_\nu.
\]

\begin{Defn}
\label{Defn:Phi-J}
For measures all of whose moments are finite, transformations $\Phi$ and $\mc{J}$ were defined in the introduction. Here are the more general definitions. $\Phi$ is the bijection
\[
\Phi : \mc{P} \rightarrow \mc{P}_{0,1}
\]
defined by
\[
F_{\Phi[\nu]}(z) = z - G_\nu(z),
\]
see \cite{Belinschi-Nica-B_t}. For $\mu \in \mc{P}_2$ with $m[\mu] = \beta$ and $\Var[\mu] = \gamma > 0$, define $\mc{J}[\mu]$ by
\[
F_\mu(z) = z - \beta - \gamma G_{\mc{J}[\mu]}(z).
\]
Then
\[
\mc{J} : \mc{P}_2 \rightarrow \mc{P},
\]
$\mc{J} \circ \Phi$ is the identity map, while $\Phi \circ \mc{J}$ is identity on $\mc{P}_{0,1}$.
\end{Defn}

\begin{Defn}
Recall that all probability measures are infinitely divisible in the Boolean sense. The Boolean-to-free Bercovici-Pata bijection (see Section~6 in \cite{BerPatDomains})
\[
\mf{B} : \mc{P} \rightarrow \mc{ID}^\boxplus
\]
is defined by
\[
\phi_{\mf{B}[\mu]}(z) = z - F_\mu(z).
\]
More generally, define the Belinschi-Nica transformations \cite{Belinschi-Nica-B_t} $\set{\mf{B}_t : t \geq 0}$ on $\mc{P}$ by
\[
\mf{B}_t[\mu] = \left( \mu^{\boxplus (1+t)} \right)^{\uplus \frac{1}{1+t}}.
\]
These transformations form a semigroup under composition, and $\mf{B}_1 = \mf{B}$.
\end{Defn}

\begin{Remark}
Note that
\[
\phi_{\mf{B}[\Phi[\rho]]}(z) = z - F_{\Phi[\rho]}(z) = G_\rho(z).
\]
So for a free convolution semigroup $\set{\mu_t : t \geq 0}$, equation~\eqref{Eq:Maassen-rep} is equivalent to
\begin{equation}
\label{Eq:Maassen-rep-free}
\mu_t = \delta_{\beta t} \boxplus \mf{B}[\Phi[\rho]]^{\boxplus \gamma t}.
\end{equation}
\end{Remark}

\begin{Defn}
For $\mu, \nu \in \mc{P}$, the subordination distribution \cite{Lenczewski-Decompositions-convolution,Nica-Subordination} $\mu \boxright \nu$ is the unique probability measure such that
\[
G_{\mu \boxplus \nu}(z) = G_\nu(F_{\mu \boxright \nu}(z)).
\]
Here $F_{\mu \boxright \nu}$ is the corresponding subordination function of $\mu \boxplus \nu$ with respect to $\nu$. If $\mu \boxright \nu \in \mc{ID}^\boxplus$, define \cite{AnsEvolution}
\[
\Phi[\mu, \nu] = \mf{B}^{-1}[\mu \boxright \nu].
\]
\end{Defn}

\begin{Lemma}
\label{Lemma:Subord-composition}
On a common domain,
\[
\phi_{\mu \boxright \nu}(z) = (\phi_\mu \circ F_\nu) (z).
\]
Also, whenever $\Phi[\mu, \nu]$ is defined,
\[
z - F_{\Phi[\mu,\nu]}(z) = (\phi_\mu \circ F_\nu)(z)
\]
and
\[
\phi_\mu = \phi_{\Phi[\mu, \nu], \nu}.
\]
\end{Lemma}

\begin{proof}
We compute
\[
\begin{split}
\phi_{\mu \boxright \nu}(z)
& = F_{\mu \boxright \nu}^{-1}(z) - z
= (F_{\mu \boxplus \nu}^{-1} \circ F_\nu)(z) - z \\
& = \bigl(\phi_{\mu \boxplus \nu}(F_\nu(z)) + F_\nu(z) \bigr) - \bigl(\phi_\nu(F_\nu(z)) + F_\nu(z) \bigr)
= (\phi_\mu \circ F_\nu) (z).
\end{split}
\]
The second property follows by combining this with the definition of $\mf{B}$. Finally,
\[
(\phi_\mu \circ F_\nu)(z) + F_{\Phi[\mu, \nu]}(z)
= z
\]
which implies the third property after comparison with equation~\eqref{Eq:Two-state-R-transform}.
\end{proof}

The following result is the analog of Corollary~4.13 in \cite{Nica-Subordination} for single-variable, unbounded distributions.

\begin{Lemma}
\label{Lemma:Subordination-ID}
If $\mu \in \mc{ID}^\boxplus$, or if $\nu = \mu \boxplus \nu'$, then $\mu \boxright \nu \in  \mc{ID}^\boxplus$.
\end{Lemma}

\begin{proof}
If $\mu \in \mc{ID}^\boxplus$, then for any $t \geq 0$,
\[
\phi_{\mu^{\boxplus t} \boxright \nu}(z) = \phi_{\mu^{\boxplus t}}(F_\nu(z)) = t \phi_\mu(F_\nu(z)) = \phi_{(\mu \boxright \nu)^{\boxplus t}}(z),
\]
and so $(\mu \boxright \nu)^{\boxplus t} = \mu^{\boxplus t} \boxright \nu$ is well defined.

If $\nu = \mu \boxplus \nu'$, then
\[
\begin{split}
\phi_{\mu \boxright \nu}(z)
& = \phi_{\mu \boxright (\mu \boxplus \nu')}(z)
= \phi_\mu(F_{\mu \boxplus \nu'}(z)) \\
& = \phi_{\mu \boxplus \nu'}(F_{\mu \boxplus \nu'}(z)) - \phi_{\nu'}(F_{\mu \boxplus \nu'}(z))
= z - F_{\mu \boxplus \nu'}(z) - \phi_{\nu'}(F_{\mu \boxplus \nu'}(z)) \\
& = z - F_{\nu'}^{-1}(F_{\mu \boxplus \nu'}(z))
= z - F_{\mu \boxright \nu'}(z)
= \phi_{\mf{B}[\mu \boxright \nu']}(z),
\end{split}
\]
and so $\mu \boxright \nu = \mf{B}[\mu \boxright \nu'] \in \mc{ID}^\boxplus$.
\end{proof}

\begin{Lemma}
\label{Lemma:Belinschi-Nica-mean}
For $(\beta, \gamma, \rho)$ a canonical triple and $t \geq 0$,
\[
\mf{B}_t[\delta_{\beta} \uplus \Phi[\rho]^{\uplus \gamma}]
= \delta_{\beta} \uplus \Phi[\rho \boxplus \delta_{\beta t} \boxplus \sigma^{\boxplus \gamma t}]^{\uplus \gamma}.
\]
\end{Lemma}

\begin{proof}
For $\gamma = 0$, the identity reduces to $\mf{B}_t[\delta_{\beta}] = \delta_\beta$. The argument for $\gamma > 0$ is a slight modification of Remark~4.4 (proof of Theorem~1.6) from \cite{Belinschi-Nica-B_t}. Following that paper, denote by $\theta$ the subordination function of $\rho \boxplus \delta_{\beta t} \boxplus \sigma^{\boxplus \gamma t}$ with respect to $\rho$, and by $\omega$ the subordination function of $(\delta_\beta \uplus \Phi[\rho]^{\uplus \gamma})^{\boxplus (t+1)}$ with respect to $(\delta_\beta \uplus \Phi[\rho]^{\uplus \gamma})$. On the one hand,
\[
G_{\rho \boxplus \delta_{\beta t} \boxplus \sigma^{\boxplus \gamma t}}(z) = G_\rho(\theta(z))
\]
and
\[
z - F_{\delta_\beta \uplus \Phi[\rho]^{\uplus \gamma}}(z) = \beta + \gamma G_\rho(z).
\]
Therefore
\begin{equation}
\label{Eq:theta-1}
\theta(z) - F_{\delta_\beta \uplus \Phi[\rho]^{\uplus \gamma}}(\theta(z)) = \beta + \gamma G_{\rho \boxplus \delta_{\beta t} \boxplus \sigma^{\boxplus \gamma t}}(z).
\end{equation}
On the other hand, denoting by $\widetilde{\theta}$ the subordination function of $\rho \boxplus \sigma^{\boxplus \gamma t}$ with respect to $\rho$, by equation~(4.8) in \cite{Belinschi-Nica-B_t},
\[
\widetilde{\theta} (z) = z - \gamma t G_{\rho \boxplus \sigma^{\boxplus \gamma t}}(z).
\]
But
\[
G_{\rho \boxplus \delta_{\beta t} \boxplus \sigma^{\boxplus \gamma t}}(z) = G_{\rho \boxplus \sigma^{\boxplus \gamma t}}(z - \beta t)
= G_\rho(\widetilde{\theta}(z - \beta t))
= G_\rho(\theta(z)).
\]
Thus
\[
\theta(z) = \widetilde{\theta}(z - \beta t)
= z - \beta t - \gamma t G_{\rho \boxplus \sigma^{\boxplus \gamma t}}(z - \beta t)
= z - \beta t - \gamma t G_{\rho \boxplus \delta_{\beta t} \boxplus \sigma^{\boxplus \gamma t}}(z).
\]
Combining this with equation~\eqref{Eq:theta-1}, we see that
\[
t \theta(z) - t F_{\delta_\beta \uplus \Phi[\rho]^{\uplus \gamma}}(\theta(z)) = z - \theta(z)
\]
and
\[
\theta(z) = \frac{1}{t+1} z + \left( 1 - \frac{1}{t+1} \right) F_{\delta_\beta \uplus \Phi[\rho]^{\uplus \gamma}}(\theta(z)).
\]
Then (see \cite{Belinschi-Nica-B_t}) it follows that $\theta = \omega$, and so the argument concludes as in that paper:
\[
\begin{split}
z - F_{\mf{B}_t[\delta_{\beta} \uplus \Phi[\rho]^{\uplus \gamma}]}(z)
& = z - \left( \left(1 - \frac{1}{t} \right) z + \frac{1}{t} \omega(z) \right) \\
& = \frac{1}{t} (z - \omega(z))
= \frac{1}{t} (z - \theta(z)) \\
& = \beta + \gamma G_{\rho \boxplus \delta_{\beta t} \boxplus \sigma^{\gamma t}}(z)
= z - F_{\delta_{\beta} \uplus \Phi[\rho \boxplus \delta_{\beta t} \boxplus \sigma^{\boxplus \gamma t}]^{\uplus \gamma}}(z). \qedhere
\end{split}
\]
\end{proof}

\section{Single-variable, complex-analytic results}
\label{Section:Complex}

\begin{Prop}
\label{Prop:Free-evolution}
For any a canonical triple $(\beta, \gamma, \rho)$, the corresponding free convolution semigroup is
\[
\mu_t = \delta_{\beta t} \uplus \Phi[\rho \boxplus \sigma_{\beta, \gamma}^{\boxplus t}]^{\uplus \gamma t}
\]
In particular, for any free convolution semigroup with non-zero, finite variance,
\[
\mc{J}[\mu_t] = \rho \boxplus \sigma_{\beta, \gamma}^{\boxplus t}.
\]
\end{Prop}

\begin{proof}
A free convolution semigroup with finite variance $\set{\mu_t}$ can be re-written as
\[
\begin{split}
\mu_t
& = \delta_{\beta t} \boxplus \mf{B}[\Phi[\rho]]^{\boxplus \gamma t} \qquad  \text{(by the Maassen representation~\eqref{Eq:Maassen-rep-free})} \\
& = \mf{B}_{t-1}[\delta_{\beta} \boxplus \mf{B}[\Phi[\rho]]^{\gamma}]^{\uplus t} \qquad \text{(by definition of $\mf{B}_{t-1}$)} \\
& = \mf{B}_t[\delta_{\beta} \uplus \Phi[\rho]^{\uplus \gamma}]^{\uplus t} \qquad \text{(by definition of $\mf{B} = \mf{B}_1$)} \\
& = \delta_{\beta t} \uplus \Phi[\rho \boxplus \delta_{\beta t} \boxplus \sigma^{\boxplus \gamma t}]^{\uplus \gamma t} \qquad \text{(by Lemma~\ref{Lemma:Belinschi-Nica-mean})} \\
& = \delta_{\beta t} \uplus \Phi[\rho \boxplus \sigma_{\beta, \gamma}^{\boxplus t}]^{\uplus \gamma t} \qquad \text{(by definition of $\sigma_{\beta, \gamma}$)}.
\end{split}
\]
For $\gamma = 0$, we have $\mu_t = \delta_{\beta t} = \sigma_{\beta t, 0}$, so the equation still holds.
\end{proof}

\begin{Lemma}
\label{Lemma:Monotone}
For a two-state free convolution semigroup $\set{(\widetilde{\mu}_t, \mu_t) : t \geq 0}$ with the relative canonical triple $(\widetilde{\beta}, \widetilde{\gamma}, \widetilde{\rho})$,
\[
\widetilde{\mu}_t = \delta_{\widetilde{\beta} t} \uplus \Phi[\widetilde{\rho} \rhd \mu_t]^{\uplus \widetilde{\gamma} t},
\]
In particular, whenever $\widetilde{\gamma} > 0$,
\[
\mc{J}[\widetilde{\mu}_t] = \widetilde{\rho} \rhd \mu_t.
\]
\end{Lemma}

\begin{proof}
Using the properties of Boolean and monotone convolutions and the definition of $\Phi$,
\[
z - F_{\delta_{\widetilde{\beta} t} \uplus \Phi[\widetilde{\rho} \rhd \mu_t]^{\uplus \widetilde{\gamma} t}}(z)
= \widetilde{\beta} t + \widetilde{\gamma} t G_{\widetilde{\rho} \rhd \mu_t}(z)
= \widetilde{\beta} t + \widetilde{\gamma} t G_{\widetilde{\rho}}(F_{\mu_t}(z))
\]
On the other hand, by formulas \eqref{Eq:Two-state-R-transform} and \eqref{Eq:Two-state-Maassen},
\[
z - F_{\widetilde{\mu}_t}(z)
= (\phi_{\widetilde{\mu}_t, \mu_{t}} \circ F_{\mu_{t}})(z)
= \widetilde{\beta} t + \widetilde{\gamma} t G_{\widetilde{\rho}}(F_{\mu_t}(z))
\]
Comparing these, we obtain the result.
\end{proof}

\begin{Thm}
\label{Thm:Two-state-free-evolution}
Fix $\widetilde{\beta} \in \mf{R}$ and $\widetilde{\gamma} > 0$.
\begin{enumerate}
\item
Let $p > 0$ and $\omega, \widetilde{\rho} \in \mc{P}$. Suppose that
\[
\omega \boxright \widetilde{\rho} \in \mc{ID}^\boxplus.
\]
Then for any $t \geq 0$,
\[
\phi_{\widetilde{\rho}} + (t/p) \phi_\omega
\]
is a Voiculescu transform of a probability measure, and so
\[
\widetilde{\rho} \boxplus \omega^{\boxplus (t/p)}
\]
is well defined.
\item
Under the assumptions of part (a), define
\[
\mu = \left( \omega \boxright \widetilde{\rho} \right)^{\boxplus (1/p)}
\]
$\mu_t = \mu^{\boxplus t}$, and
\[
\widetilde{\mu}_t = \delta_{\widetilde{\beta} t} \uplus \Phi[\widetilde{\rho} \boxplus \omega^{\boxplus (t/p)}]^{\uplus \widetilde{\gamma} t}.
\]
Then $\widetilde{\mu}_1$ has finite, non-zero variance, the family $\set{(\widetilde{\mu}_t, \mu_t) : t \geq 0}$ form a two-state free convolution semigroup, and
\[
\mc{J}[\widetilde{\mu}_t] = \widetilde{\rho} \boxplus \omega^{\boxplus (t/p)}.
\]
\item
Conversely, let $\set{(\widetilde{\mu}_t, \mu_t) : t \geq 0}$ be a general two-state free convolution semigroup of \emph{compactly supported} measures such that $\widetilde{\mu}_1$ has non-zero variance. Then there exist $p > 0$ and $\omega , \widetilde{\rho} \in \mc{P}$ such that $\omega \boxright \widetilde{\rho} \in \mc{ID}^\boxplus$, $\widetilde{\rho} \boxplus \omega^{\boxplus (t/p)}$ is well defined (in the sense of part (a)) for all $t \geq 0$, and the relations in part (b) hold.
\end{enumerate}
\end{Thm}

\begin{proof}
For part (a), using Lemma~\ref{Lemma:Subord-composition},
\begin{equation}
\label{Eq:rho-monotone}
\begin{split}
\phi_{\widetilde{\rho}}(z) + (t/p) \phi_\omega(z)
& = \phi_{\widetilde{\rho}}(z) + (t/p)\phi_{\omega \boxright \widetilde{\rho}}\bigl(F^{-1}_{\widetilde{\rho}}(z) \bigr)
= z - F^{-1}_{\widetilde{\rho}}(z) + \phi_{(\omega \boxright \widetilde{\rho})^{\boxplus (t/p)}}\bigl(F^{-1}_{\widetilde{\rho}}(z) \bigr) \\
& = z - F^{-1}_{(\omega \boxright \widetilde{\rho})^{\boxplus (t/p)}}\bigl(F^{-1}_{\widetilde{\rho}}(z) \bigr)
= z - F^{-1}_{\widetilde{\rho} \rhd (\omega \boxright \widetilde{\rho})^{\boxplus (t/p)}}(z)
= \phi_{\widetilde{\rho} \rhd (\omega \boxright \widetilde{\rho})^{\boxplus (t/p)}}(z).
\end{split}
\end{equation}
Since the monotone convolution is known to preserve positivity, this implies part (a). Next, it is clear that in part (b), $\set{\mu_t : t \geq 0}$ form a free convolution semigroup. From equation~\eqref{Eq:rho-monotone}, it follows that
\[
\widetilde{\rho} \boxplus \omega^{\boxplus (t/p)} = \widetilde{\rho} \rhd \mu_t.
\]
Part (b) now follows from Lemma~\ref{Lemma:Monotone}.

To prove part (c), first suppose that $\mu_t = \delta_{\beta t}$ has zero variance. Then all the relations hold if we set $p=1$, $\omega = \delta_\beta$, and $\widetilde{\rho}$ to be the measure in the relative canonical triple of $\set{(\widetilde{\mu}_t, \mu_t) : t \geq 0}$.

In the remainder of the argument, we assume that $\mu_1$ has non-zero variance, and use results from Section~\ref{Section:Algebraic}. Let $\widetilde{\rho}$ be the measure in the relative canonical triple of $\set{(\widetilde{\mu}_t, \mu_t) : t \geq 0}$. By Corollary~\ref{Cor:Compactly-supported}, there exists a linear, unital, not necessarily positive functional $\tau$ such that
\begin{equation}
\label{Eq:mu}
\tau \boxright \widetilde{\rho} = \mu
\end{equation}
is $\boxplus$-infinitely divisible, and
\[
\widetilde{\rho} \boxplus \tau^{\boxplus p} = \mc{J}[\widetilde{\mu}_t]
\]
can be identified with a positive measure. Moreover, it follows from equation~\eqref{Eq:mu} that $\Var[\tau] = \Var[\mu] > 0$. So by Lemma~\ref{Lemma:Hari}, for sufficiently large $p$, $\omega = \tau^{\boxplus p}$ can itself be identified with a positive measure. The result follows.
\end{proof}

The following corollary is an immediate consequence of Lemma~\ref{Lemma:Subordination-ID}.

\begin{Cor}
\label{Cor}
The assumptions of parts (a,b) of the Theorem are satisfied in the following two cases.
\begin{enumerate}
\item
$\widetilde{\rho} \in \mc{P}$ is arbitrary and $\omega = \tau \in \mc{ID}^\boxplus$. In this case one can, without loss of generality, take $p=1$.
\item
$\omega \in \mc{P}$ is arbitrary, and $\widetilde{\rho} = \nu \boxplus \omega$ for some $\nu \in \mc{P}$.
\end{enumerate}
In particular, for any $\widetilde{\rho} \in \mc{P}$ and $\tau \in \mc{ID}^\boxplus$, there exists a two-state free convolution semigroup $\set{(\widetilde{\mu}_t, \mu_t) : t \geq 0}$ such that
\[
\mc{J}[\widetilde{\mu}_t] = \widetilde{\rho} \boxplus \tau^{\boxplus t}.
\]
\end{Cor}

I am grateful to Serban Belinschi for a discussion leading to the following example.

\begin{Example}
\label{Example:Counterexample}
Recall that the analytic $R$-transform is defined by $R_\mu(z) = \phi_\mu(1/z)$. Let
\[
\tau_\eps = \frac{1}{2} (\delta_{-\eps} + \delta_\eps).
\]
Then
\[
R_{\tau_\eps}(z) = \frac{2 \eps^2 z}{\sqrt{1 + 4 \eps^2 z^2} + 1}
\]
is analytic for $\abs{z} < (2 \eps)^{-1}$ and grows as $\abs{R_{\tau_\eps}(z)} \approx \eps^2 \abs{z}$. It follows from Theorem~2 of \cite{BerVoiSuperconvergence} that for sufficiently small $\eps$,
\[
z + t R_{\tau_\eps}(z)
\]
is an $R$-transform of a positive measure for all $t \in [0,1]$. On the other hand, $\tau_\eps^{\boxplus t}$ is well-defined and positive for all $t \geq 1$. It follows that $\sigma \boxplus \tau_\eps^{\boxplus t}$ is well-defined and positive for all $t \geq 0$. However, $\tau_\eps \not \in \mc{ID}^{\boxplus}$, and $\sigma \neq \nu \boxplus \tau_\eps^{\boxplus p}$ for any $p > 0$, so this family is not covered by the preceding corollary. Nevertheless, $F_\sigma(\mf{C}^+) = \mf{C}^+ \setminus \set{z : \abs{z} \leq 2}$, and
\[
\phi_{\tau_\eps}(z) = \frac{\sqrt{z^2 + 4 \eps^2} - z}{2}
\]
is analytic on this image for $\eps < 1$. It follows that
\[
\phi_{\tau_\eps \boxright \sigma} = \phi_{\tau_\eps} \circ F_\sigma
\]
analytically extends to $\mf{C}^+$, and so $\tau_\eps \boxright \sigma \in \mc{ID}^\boxplus$. So this family is still covered by the preceding theorem.
\end{Example}

\begin{Question}
Can the hypothesis in parts (a,b) of the Theorem be weakened to the assumption in the following proposition? In other words, does this assumption imply that the (equivalent) statements in the following proposition necessarily hold?
\end{Question}

\begin{Prop}
Let $\widetilde{\rho}\in \mc{P}$, $\tau \in \mc{P}$, and suppose that $\widetilde{\rho} \boxplus \tau^{\boxplus t}$ is defined for all $t \geq 0$. The following are equivalent.
\begin{enumerate}
\item
$\tau \boxright \widetilde{\rho} \in \mc{ID}^\boxplus$.
\item
$F_{\widetilde{\rho} \boxplus \tau^{\boxplus t}}$ is subordinate to $F_{\widetilde{\rho}}$ for all $t \geq 0$, in the sense that there exist analytic transformations $\theta_t : \mf{C}^+ \rightarrow \mf{C}^+$ such that $F_{\widetilde{\rho} \boxplus \tau^{\boxplus t}}(z) = F_{\widetilde{\rho}}(\theta_t(z))$.
\item
$\set{\Phi[\widetilde{\rho} \boxplus \tau^{\boxplus t}] : t \geq 0}$ is the first component of a two-state free convolution semigroup.
\end{enumerate}
\end{Prop}

\begin{proof}
Calculations in the proof of the theorem show that if $\theta_t$ exists, then $\theta_t = F_{(\tau \boxright \widetilde{\rho})^{\boxplus t}}$. This shows that (a $\Leftrightarrow$ b). The same calculations also imply that if $\set{(\widetilde{\mu}_t = \Phi[\widetilde{\rho} \boxplus \tau^{\boxplus t}], \mu_t) : t \geq 0}$ is a two-state free convolution semigroup, then $\mu_t = (\tau \boxright \widetilde{\rho})^{\boxplus t}$. Thus (a $\Leftrightarrow$ c).
\end{proof}

\begin{Lemma}
\label{Lemma:Basic-properties}
Subordination distributions have the following properties.
\[
(\mu \boxplus \nu) \boxright \rho = (\mu \boxright \rho) \boxplus (\nu \boxright \rho).
\]
\[
\sigma \boxright \mu = \mf{B}[\Phi[\mu]].
\]
\[
\mu \boxright \delta_0 = \mu.
\]
\[
\mu \boxright \mu = \mf{B}[\mu].
\]
\[
\delta_a \boxright \mu = \delta_a.
\]
There is a corresponding list of properties for $\Phi[\cdot, \cdot]$.
\end{Lemma}

\begin{proof}
All these properties follow immediately from
\[
\phi_{\mu \boxright \nu}(z) = (\phi_\mu \circ F_\nu) (z). \qedhere
\]
\end{proof}

In a number of the following examples, free convolution semigroups $\set{\mu_t : t \geq 0}$ will have finite variance, and so will be associated with canonical triples $(\beta, \gamma, \rho)$; in all cases, the relative canonical triple is $(\widetilde{\beta}, \widetilde{\gamma}, \widetilde{\rho})$.

\begin{Example}
Let $\widetilde{\rho} = \rho \in \mc{P}$. Then the first component of the corresponding two-state free convolution semigroup satisfies
\[
\mc{J}[\widetilde{\mu}_t] = \widetilde{\rho} \boxplus \sigma_{\beta, \gamma}^{\boxplus t},
\]
so that in Corollary~\ref{Cor}, $\tau = \sigma_{\beta, \gamma} \in \mc{ID}^{\boxplus}$ is a semicircular distribution. Indeed,
\[
\sigma_{\beta, \gamma} \boxright \widetilde{\rho}
= (\delta_{\beta} \boxplus \sigma^{\boxplus \gamma}) \boxright \widetilde{\rho}
= \delta_\beta \boxplus \mf{B}[\Phi[\widetilde{\rho}]]^{\boxplus \gamma}
= \mu.
\]
In the particular case when $\widetilde{\beta} = \beta$ and $\widetilde{\gamma} = \gamma$, it follows that $\widetilde{\mu}_t = \mu_t$ form a free convolution semigroup, and we are in the Belinschi-Nica setting of Proposition~\ref{Prop:Free-evolution}.
\end{Example}

\begin{Example}
Let $\widetilde{\rho} = \delta_0$ and $\rho \in \mc{P}$. Then the first component of the corresponding two-state free convolution semigroup satisfies
\[
\mc{J}[\widetilde{\mu}_t] = \mu_t,
\]
so that in Corollary~\ref{Cor}, $\tau = \mu \in \mc{ID}^{\boxplus}$ is arbitrary. Indeed,
\[
\mu \boxright \delta_0 = \mu.
\]
These are the (distributions of) two-state free Brownian motions (in \cite{Ans-Two-Brownian}, they were called algebraic two-state free Brownian motions).
\end{Example}

\begin{Example}
Let $\widetilde{\rho} \in \mc{P}$ and $\gamma = 0$, so that $\mu_t = \delta_{\beta t}$. Then the first component of the corresponding two-state free convolution semigroup satisfies
\[
\mc{J}[\widetilde{\mu}_t] = \widetilde{\rho} \boxplus \delta_{\beta t},
\]
so that in Corollary~\ref{Cor}, $\tau = \delta_\beta \in \mc{ID}^{\boxplus}$. Indeed,
\[
\delta_\beta \boxright \widetilde{\rho} = \delta_\beta = \mu.
\]
For general $\beta$ and measures all of whose moments are finite,
\[
\widetilde{\mu}_t = \delta_{\widetilde{\beta} t} \uplus \Phi[\widetilde{\rho} \boxplus \delta_{\beta t}]^{\uplus \widetilde{\gamma} t}.
\]
are precisely the families constructed in Proposition~7 of \cite{Ans-Mlot-Semigroups}. For $\beta = 0$, this is a Boolean convolution semigroup, and an arbitrary Boolean convolution semigroup (with finite variance) arises in this way.

On the other hand, if $\widetilde{\mu}_t = \delta_{\widetilde{\beta} t}$, for any free convolution semigroup $\set{\mu_t : t \geq 0}$ the measures $(\delta_{\widetilde{\beta} t}, \mu_t)$ form a two-state free convolution semigroup.
\end{Example}

\begin{Example}
\label{Example:General-b}
Let $\widetilde{\rho}, \rho \in \mc{P}$ such that $\mc{J}[\widetilde{\rho}] = \rho$. That is, for some $\widetilde{b}$ and $\widetilde{c} > 0$,
\[
\widetilde{\rho} = \delta_{\widetilde{b}} \uplus \Phi[\rho]^{\uplus \widetilde{c}}.
\]
Denote
\[
p = \widetilde{c}/\gamma, \quad u = \widetilde{b} - \beta \widetilde{c}/\gamma.
\]
Then the first component of the corresponding two-state free convolution semigroup satisfies
\[
\mc{J}[\widetilde{\mu}_t] = \widetilde{\rho} \boxplus \omega^{\boxplus p t},
\]
where in Corollary~\ref{Cor}, $\omega = \delta_{-u} \boxplus \widetilde{\rho} \in \mc{P}$ but in general is not freely infinitely divisible. Indeed,
\[
\begin{split}
\left( (\delta_{-u} \boxplus \widetilde{\rho}) \boxright \widetilde{\rho} \right)^{\boxplus (1/p)}
& = \delta_{-(u/p)} \boxplus \mf{B}[\widetilde{\rho}]^{\boxplus (1/p)}
= \mf{B}[\delta_{-(u/p)} \uplus \widetilde{\rho}^{\uplus (1/p)}] \\
& = \mf{B}[\delta_{-(u/p)} \uplus \left( \delta_{\widetilde{b}} \uplus \Phi[\rho]^{\uplus \widetilde{c}}\right) ^{\uplus (1/p)}]
= \mf{B}[\delta_{\beta} \uplus \Phi[\rho]^{\uplus \gamma}]
= \mu.
\end{split}
\]
If $\tau = \omega^{\boxplus (\gamma/\widetilde{c})} \in \mc{P}$, then
\[
\mc{J}[\widetilde{\mu}_t] = \widetilde{\rho} \boxplus \tau^{\boxplus t}.
\]
\end{Example}

\begin{Remark}
A free Meixner distribution $\mu_{b,c,\beta,\gamma}$ with parameters $b, \beta \in \mf{R}$, $c+\gamma, \gamma \geq 0$ is the probability measure with Jacobi parameters
\[
J(\mu_{b,c,\beta,\gamma}) =
\begin{pmatrix}
\beta, & b + \beta, & b + \beta, & b + \beta, & \ldots \\
\gamma, & c + \gamma, & c + \gamma, & c + \gamma, & \ldots
\end{pmatrix}.
\]
For other values of $c, \gamma$, these Jacobi parameters determine a unital, linear, but not positive definite functional. Normalized free Meixner distributions $\mu_{b,c} = \mu_{b, c, 0, 1}$ have mean zero and variance $1$, and are positive for $c \geq -1$.

Free Meixner distributions form a two-parameter semigroup with respect to $\boxplus$:
\[
\mu_{b,c,\beta,\gamma} \boxplus \mu_{b,c,\beta',\gamma'} = \mu_{b,c,\beta + \beta', \gamma + \gamma'},
\]
see Definition~2 of \cite{Ans-Mlot-Semigroups}. In particular,
\[
\mu_{b,c,\beta,\gamma}^{\boxplus t} = \mu_{b, c, \beta t, \gamma t}.
\]
Also,
\[
\mf{B}_t[\mu_{b,c,\beta,\gamma}] = \mu_{b + \beta t, c + \gamma t, \beta, \gamma}.
\]
\end{Remark}

\begin{Lemma}
The subordination distribution of two Meixner distributions with special parameters
\[
\mu_{b,c,\beta',\gamma'} \boxright \mu_{b,c,\beta,\gamma} = \mu_{b + \beta, c + \gamma, \beta', \gamma'}
\]
is again a free Meixner distribution.
\end{Lemma}

\begin{proof}
Using Lemma~\ref{Lemma:Basic-properties} and the properties from the preceding remark, we compute
\[
\begin{split}
\mu_{b,c,\beta',\gamma'} \boxright \mu_{b,c,\beta,\gamma}
& = \Bigl( \delta_{\beta' - \beta \gamma'/\gamma} \boxplus \mu_{b,c,\beta,\gamma}^{\boxplus (\gamma'/\gamma)} \Bigr) \boxright \mu_{b,c,\beta,\gamma}
= \delta_{\beta' - \beta \gamma'/\gamma} \boxplus \Bigl( \mu_{b,c,\beta,\gamma} \boxright \mu_{b,c,\beta,\gamma} \Bigr)^{\boxplus (\gamma'/\gamma)} \\
& = \delta_{\beta' - \beta \gamma'/\gamma} \boxplus \mf{B}[\mu_{b,c,\beta,\gamma}]^{\boxplus (\gamma'/\gamma)}
= \delta_{\beta' - \beta \gamma'/\gamma} \boxplus \mu_{b + \beta,c + \gamma,\beta,\gamma}^{\boxplus (\gamma'/\gamma)} \\
& = \delta_{\beta' - \beta \gamma'/\gamma} \boxplus \mu_{b + \beta,c + \gamma,\beta \gamma'/\gamma,\gamma'}
= \mu_{b + \beta,c + \gamma,\beta',\gamma'}.
\end{split}
\]
\end{proof}

\begin{Remark}
Since $\nu \rhd (\mu \boxright \nu) = \mu \boxplus \nu$, the preceding lemma implies a monotone convolution identity
\[
\mu_{b,c,\beta,\gamma} \rhd \mu_{b + \beta, c + \gamma, \beta', \gamma'} = \mu_{b, c, \beta + \beta', \gamma + \gamma'}.
\]
This result can also be proved directly using the $F$-transforms, but the computation is rather surprising. Since $\mu_{0, 0, \beta, \gamma}$ are semicircular, $\mu_{b, 0, \beta, \gamma}$ free Poisson, $\mu_{b, -\gamma, \beta, \gamma}$ Bernoulli, and $\mu_{0, -\gamma, 0, 2 \gamma}$ arcsine distributions, we get various identities between them involving the monotone convolution. For example,
\[
\mu_{b,c} \rhd \mu_{b, c + 1} = \mu_{b, c}^{\boxplus 2},
\]
which for $b=0$, $c=-1$ says Bernoulli $\rhd$ Semicircle $=$ Arcsine. See \cite{Mlotkowski-Fuss-Catalan} for related results.
\end{Remark}

\begin{Example}
For a particular case of Example~\ref{Example:General-b}, let $\widetilde{c} > 0$ and $c \geq 0$. The two-state free Meixner semigroups from \cite{Ans-Mlot-Semigroups} satisfy
\[
J(\widetilde{\mu}_t) =
\begin{pmatrix}
\widetilde{\beta} t, & \widetilde{b} + \beta t, & b + \beta t, & b + \beta t, & \ldots \\
\widetilde{\gamma} t, & \widetilde{c} + \gamma t, & c + \gamma t, & c + \gamma t, & \ldots
\end{pmatrix}
\]
and
\[
J(\mu_t) =
\begin{pmatrix}
\beta t, & b + \beta t, & b + \beta t, & \ldots \\
\gamma t, & c + \gamma t, & c + \gamma t, & \ldots
\end{pmatrix}.
\]
Thus
\[
J(\mc{J}[\widetilde{\mu}_t]) =
\begin{pmatrix}
\widetilde{b} + \beta t, & b + \beta t, & b + \beta t, & \ldots \\
\widetilde{c} + \gamma t, & c + \gamma t, & c + \gamma t, & \ldots
\end{pmatrix},
\]
so $\mc{J}[\widetilde{\mu}_t] = \widetilde{\rho} \boxplus \omega^{\boxplus (\gamma/\widetilde{c}) t}$, where
\[
J(\widetilde{\rho}) =
\begin{pmatrix}
\widetilde{b}, & b, & b, & b, & \ldots \\
\widetilde{c}, & c, & c, & c, & \ldots
\end{pmatrix}.
\]
and
\[
J(\omega) =
\begin{pmatrix}
\beta \widetilde{c}/\gamma, & \beta \widetilde{c}/\gamma + b - \widetilde{b}, & \beta \widetilde{c}/\gamma + b - \widetilde{b}, & \beta \widetilde{c}/\gamma + b - \widetilde{b}, & \ldots \\
\widetilde{c}, & c, & c, & c, & \ldots
\end{pmatrix}.
\]
In particular, $\widetilde{\rho} = \delta_{\widetilde{b} - \beta \widetilde{c}/\gamma} \boxplus \omega$. Note that both $\widetilde{\rho}$ and $\omega$ are free Meixner distributions. Also,
\[
J(\rho) =
\begin{pmatrix}
b, & b, & b, & b, & \ldots \\
c, & c, & c, & c, & \ldots
\end{pmatrix},
\]
so $\rho = \sigma_{b,c} = \delta_b \boxplus \sigma^{\boxplus c}$ and $\mc{J}[\widetilde{\rho}] = \rho$. Finally, for $\tau = \omega^{\boxplus (\gamma/\widetilde{c})}$,
\[
J(\tau) =
\begin{pmatrix}
\beta, & \beta + b - \widetilde{b}, & \beta + b - \widetilde{b}, & \beta + b - \widetilde{b}, & \ldots \\
\gamma, & \gamma + c - \widetilde{c}, & \gamma + c - \widetilde{c}, & \gamma + c - \widetilde{c}, & \ldots
\end{pmatrix}.
\]
So $\tau \in \mc{ID}^{\boxplus}$ for $c \geq \widetilde{c}$, $\tau \in \mc{P}$ for $\gamma + c \geq \widetilde{c}$, and for $\gamma + c <\widetilde{c}$, $\tau$ is not a positive functional.
\end{Example}

\begin{Prop}
\label{Prop:PDE}
Let $(\widetilde{\mu}_t, \mu_t)$ be a general two-state free convolution semigroup. Then we have two evolution equations
\begin{equation}
\label{Eq:PDE-2-state}
\partial_t F_{\widetilde{\mu}_t}
= \phi_\mu(F_{\mu_t}) - \phi_{\widetilde{\mu}, \mu}(F_{\mu_t}) - \phi_\mu(F_{\mu_t}) \partial_z F_{\widetilde{\mu}_t}
\end{equation}
and
\[
\partial_t F_{\mu_t}
= - \phi_\mu(F_{\mu_t}) \partial_z F_{\mu_t}.
\]
\end{Prop}

\begin{proof}
The second equation is standard, see equation~(3.18) in \cite{VDN}. Using~\eqref{Eq:Two-state-R-transform},
\[
\partial_t F_{\widetilde{\mu}_t} + \phi_{\widetilde{\mu}, \mu}(F_{\mu_t}) + t \phi_{\widetilde{\mu}, \mu}'(F_{\mu_t}) \partial_t F_{\mu_t}
= \partial_t F_{\widetilde{\mu}_t} + \phi_{\widetilde{\mu}, \mu}(F_{\mu_t}) - t \phi_{\widetilde{\mu}, \mu}'(F_{\mu_t}) \phi_\mu(F_{\mu_t}) \partial_z F_{\mu_t} = 0
\]
and
\[
\partial_z F_{\widetilde{\mu}_t} + t \phi_{\widetilde{\mu}, \mu}'(F_{\mu_t}) \partial_z F_{\mu_t} = 1.
\]
Plugging in, we get
\[
\begin{split}
\partial_t F_{\widetilde{\mu}_t}
& = - \phi_{\widetilde{\mu}, \mu}(F_{\mu_t}) + \frac{1 - \partial_z F_{\widetilde{\mu}_t}}{\phi_{\widetilde{\mu}, \mu}'(F_{\mu_t}) \partial_z F_{\mu_t}} \phi_{\widetilde{\mu}, \mu}'(F_{\mu_t}) \phi_\mu(F_{\mu_t}) \partial_z F_{\mu_t} \\
& = \phi_\mu(F_{\mu_t}) - \phi_{\widetilde{\mu}, \mu}(F_{\mu_t}) - \phi_\mu(F_{\mu_t}) \partial_z F_{\widetilde{\mu}_t}. \qedhere
\end{split}
\]
\end{proof}

\begin{Defn}
The functional $L_t$ is the \emph{generator} of the family $\set{\mu_t : t \geq 0}$ of functionals at time $t$, with domain $\mc{D}$, if for any $f \in \mc{D}$
\[
\ip{L_t}{f} = \frac{d}{dt}\ip{\mu_t}{f}.
\]
\end{Defn}

\begin{Prop}
Let $(\widetilde{\mu}_t, \mu_t)$ be a general two-state free convolution semigroup with finite variance, with canonical triples $\set{(\widetilde{\beta}, \widetilde{\gamma}, \widetilde{\rho}), (\beta, \gamma, \rho)}$. Denote $\mc{J}[\widetilde{\mu}_t] = \widetilde{\nu}_t$, $\mc{J}[\mu_t] = \nu_t$. Note that
\[
\nu_t = \rho \boxplus \sigma_{\beta, \gamma}^{\boxplus t},
\]
and for measures covered in Theorem~\ref{Thm:Two-state-free-evolution}, $\widetilde{\nu}_t = \widetilde{\rho} \boxplus \tau^{\boxplus t}$.

Then the generators of the families $\set{\widetilde{\mu}_t}$ and $\set{\mu_t}$ with domain
\[
\mc{D} = \Span{\set{\frac{1}{z - x} : z \in \mf{C} \setminus \mf{R}}}
\]
are, respectively,
\[
\begin{split}
\widetilde{L}_t
& = \widetilde{\gamma} (\widetilde{\mu}_t \otimes \widetilde{\mu}_t \otimes \widetilde{\nu}_t) \partial^2
- \gamma (\widetilde{\mu}_t \otimes \widetilde{\mu}_t \otimes \nu_t) \partial^2 \\
&\quad + (\widetilde{\beta} - \beta) (\widetilde{\mu}_t \otimes \widetilde{\mu}_t) \partial
+ \gamma (\widetilde{\mu}_t \otimes \nu_t) (\partial_x \otimes 1) \partial + \beta \widetilde{\mu}_t \partial_x
\end{split}
\]
and
\[
L_t = \gamma (\mu_t \otimes \nu_t) (\partial_x \otimes 1) \partial + \beta \mu_t \partial_x.
\]
Here $\partial: \mc{D} \rightarrow \mc{D} \otimes \mc{D}$ is the difference quotient operation,
\[
(\partial f)(x,y) = \frac{f(x) - f(y)}{x - y}.
\]
\end{Prop}

\begin{proof}
Note first that
\[
(\phi_{\widetilde{\mu}, \mu} \circ F_{\mu_t})(z)
= \frac{1}{t} (\phi_{\widetilde{\mu}_t, \mu_t} \circ F_{\mu_t})(z)
= \frac{1}{t} (z - F_{\widetilde{\mu}_t}(z))
= \widetilde{\beta} + \widetilde{\gamma} G_{\widetilde{\nu}_t}(z)
\]
and similarly $\phi_{\mu} \circ F_{\mu_t} = \beta + \gamma G_{\nu_t}$. Therefore in this case, equation~\eqref{Eq:PDE-2-state} says
\[
\partial_t F_{\widetilde{\mu}_t}
= \beta + \gamma G_{\nu_t} - \widetilde{\beta} + \widetilde{\gamma} G_{\widetilde{\nu}_t} - (\beta + \gamma G_{\nu_t}) \partial_z F_{\widetilde{\mu}_t}.
\]
Equivalently,
\[
\partial_t G_{\widetilde{\mu}_t}
= - (\beta + \gamma G_{\nu_t} - \widetilde{\beta} + \widetilde{\gamma} G_{\widetilde{\nu}_t}) G_{\widetilde{\mu}_t}^2 - (\beta + \gamma G_{\nu_t}) \partial_z G_{\widetilde{\mu}_t}.
\]
In other words,
\[
\begin{split}
\partial_t \ip{\widetilde{\mu}_t}{\frac{1}{z - x}}
& = - \gamma \ip{\widetilde{\mu}_t \otimes \widetilde{\mu}_t \otimes \nu_t}{\partial^2 \frac{1}{z - x}}
- \beta \ip{\widetilde{\mu}_t \otimes \widetilde{\mu}_t}{\partial \frac{1}{z - x}} \\
&\quad + \widetilde{\gamma} \ip{\widetilde{\mu}_t \otimes \widetilde{\mu}_t \otimes \widetilde{\nu}_t}{\partial^2 \frac{1}{z - x}}
+ \widetilde{\beta} \ip{\widetilde{\mu}_t \otimes \widetilde{\mu}_t}{\partial \frac{1}{z - x}} \\
&\quad + \gamma \ip{\widetilde{\mu}_t \otimes \nu_t}{(\partial_x \otimes 1) \partial \frac{1}{z - x}}
+ \beta \ip{\widetilde{\mu}_t}{\partial_x \frac{1}{z - x}}
\end{split}
\]
The formula for the generator $\widetilde{L}_t$ of $\set{\widetilde{\mu}_t}$ on the span of such functions follows. The formula for $L_t$ follows by setting $\widetilde{\mu}_t = \mu_t$.
\end{proof}

\begin{Remark}
Setting $t=0$ in the preceding proposition, $\mu_0 = \widetilde{\mu}_0 = \delta_0$, $\widetilde{\nu}_0 = \widetilde{\rho}$, and $\nu_0 = \rho$. Thus
\[
L_0 f = \gamma \ip{\delta_0 \otimes \rho}{(\partial_x \otimes 1) \partial f}
+ \beta \ip{\delta_0}{\partial_x f}
= \gamma \int_{\mf{R}} \frac{f(y) - f(0) - y f'(0)}{y^2} \,d\rho(y) + \beta f'(0)
\]
and
\[
\begin{split}
\widetilde{L}_0 f
& = \widetilde{\gamma} \ip{\delta_0 \otimes \delta_0 \otimes \widetilde{\rho}}{\partial^2 f}
- \gamma \ip{\delta_0 \otimes \delta_0 \otimes \rho}{\partial^2 f}
+ (\widetilde{\beta} - \beta) \ip{\delta_0 \otimes \delta_0}{\partial f} \\
&\quad + \gamma \ip{\delta_0 \otimes \rho}{(\partial_x \otimes 1) \partial f}
+ \beta \ip{\delta_0}{\partial_x f} \\
& = \int_{\mf{R}} \frac{f(y) - f(0) - y f'(0)}{y^2} \,d(\widetilde{\gamma} \widetilde{\rho} - \gamma \rho)(y) + (\widetilde{\beta} - \beta) f'(0) \\
&\quad + \int_{\mf{R}} \frac{f(y) - f(0) - y f'(0)}{y^2} \,d(\gamma \rho)(y) + \beta f'(0) \\
& = \widetilde{\gamma} \int_{\mf{R}} \frac{f(y) - f(0) - y f'(0)}{y^2} \,d\widetilde{\rho}(y) + \widetilde{\beta} f'(0)
\end{split}
\]
has exactly the same form as in Proposition~3 of \cite{Ans-Generator}; see also Remark~11 of that paper.
\end{Remark}

\begin{Remark}
Boolean evolution corresponds to $\beta = \gamma = 0$, $\mu_t = \delta_0$. Then
\[
\partial_t F_{\widetilde{\mu}_t}(z) = - \phi_{\widetilde{\mu}, \delta_0}(z).
\]
In fact, since $\phi_{\widetilde{\mu}, \delta_0}(z) = z - F_{\widetilde{\mu}}(z)$, this is easy to see directly. It follows that in this case,
\[
\widetilde{L}_t
= \widetilde{\gamma} (\widetilde{\mu}_t \otimes \widetilde{\mu}_t \otimes \widetilde{\nu}_t) \partial^2
+ \widetilde{\beta} (\widetilde{\mu}_t \otimes \widetilde{\mu}_t) \partial.
\]
For $t=0$, we again get the formula from the preceding remark.

Similarly, distributions of analytic two-state free Brownian motions correspond to $\widetilde{\gamma} = \gamma = 1$, $\widetilde{\beta} = 0$, $\widetilde{\mu} = \Phi[\mu]$ and $\mu = \delta_\beta \boxplus \sigma$, so that $\widetilde{\nu}_t = \nu_t = \mu_t$. Then the generator formula reduces to
\[
\begin{split}
\widetilde{L}_t
& = - \beta (\widetilde{\mu}_t \otimes \widetilde{\mu}_t) \partial
+ (\widetilde{\mu}_t \otimes \mu_t) (\partial_x \otimes 1) \partial + \beta \widetilde{\mu}_t \partial_x \\
& = \widetilde{\mu}_t \Bigl( - \beta (1 \otimes \widetilde{\mu}_t) \partial
+ \partial_x (1 \otimes \mu_t) \partial + \beta \partial_x \Bigr),
\end{split}
\]
consistent with the result of Proposition~24 in \cite{Ans-Generator}.
\end{Remark}

\section{Background II}

\subsection{Multivariate polynomials}

The number $d \in \mf{N}$ will be fixed throughout the remainder of the article. Denote
\[
\mb{x} = (x_1, x_2, \ldots, x_d)
\]
a $d$-tuple of variables, and similarly for $\mb{z}$, etc. Let
\[
\mf{C} \langle \mb{x} \rangle = \mf{C} \langle x_1, x_2, \ldots, x_d \rangle
\]
be the algebra of polynomials in $d$ non-commuting variables. For $k \geq 1$ and
\[
\vec{u} = (u(1), u(2), \ldots, u(k)) \in \set{1, \ldots, d}^k
\]
a multi-index, denote
\[
x_{\vec{u}} = x_{u(1)} x_{u(2)} \ldots x_{u(k)}.
\]
Denote
\[
\Dist(d) = \set{\mu : \mf{C} \langle x_1, x_2, \ldots, x_d \rangle \rightarrow \mf{C} \text{ unital, linear functionals}},
\]
For $\beta \in \mf{R}^d$, the element $\delta_{\beta} \in \Dist(d)$ is
\[
\delta_\beta[x_{\vec{u}}] = \beta_{\vec{u}}.
\]

\subsection{Free, Boolean, and two-state free convolutions}

Let $\mu \in \Dist(d)$. Denote its moment generating function by
\[
M^\mu(\mb{z}) = \sum_{\vec{u}} \mu[x_{\vec{u}}] z_{\vec{u}}.
\]
The (combinatorial) $R$-transform $R^\mu$ of $\mu$ is determined by
\begin{equation}
\label{Eq:M-R}
R^\mu \Bigl( z_1 (1 + M^\mu(\mb{z})), \ldots, z_d (1 + M^\mu(\mb{z})) \Bigr) = M^\mu(\mb{z}),
\end{equation}
see Lecture~16 of \cite{Nica-Speicher-book}. The free convolution of two functionals $\mu \boxplus \nu$ is determined by the equality
\[
R^{\mu \boxplus \nu} = R^\mu + R^\nu.
\]
In the algebraic setting, any functional is $\boxplus$-infinitely divisible.

Similarly, the $\eta$-transform $\eta^\mu$ is
\[
\eta^\mu(\mb{z}) = (1 + M^\mu(\mb{z}))^{-1} M^\mu(\mb{z})
\]
(for a multivariate power series $F$, $F^{-1}$ will denote its multiplicative inverse). The Boolean convolution of two functionals $\mu \uplus \nu$ is determined by the equality
\[
\eta^{\mu \uplus \nu} = \eta^\mu + \eta^\nu.
\]
Finally, for $\widetilde{\mu}, \mu \in \Dist(d)$, the two-state $R$-transform $R^{\widetilde{\mu}, \mu}$ is determined by
\[
\eta^{\widetilde{\mu}}(\mb{z}) = R^{\widetilde{\mu}, \mu} \Bigl( z_1 (1 + M^\mu(\mb{z})), \ldots, z_d (1 + M^{\mu}(\mb{z})) \Bigr) (1 + M^\mu(\mb{z}))^{-1},
\]
and the two-state free convolution of two pairs of functionals
\[
(\rho, \mu \boxplus \nu) = (\widetilde{\mu},\mu) \boxplus_c (\widetilde{\nu}, \nu)
\]
is determined by the equality
\[
R^{\rho, \mu \boxplus \nu} = R^{\widetilde{\mu},\mu} + R^{\widetilde{\nu}, \nu}.
\]
See Section~2.5 of \cite{AnsEvolution}.

If $d=1$ and $\mu$ is a compactly supported probability measure on $\mf{R}$, it can be identified with an element of $\Dist(1)$. In this case, the complex function transforms from Section~\ref{Section:Preliminaries} have power series expansions related to the power series from this section by
\[
1 + M^\mu(z) = \frac{1}{z} G_\mu \left( \frac{1}{z} \right), \quad
R^\mu(z) = z R_\mu(z) = z \phi_\mu \left( \frac{1}{z} \right),
\]
\[
\eta^\mu(z) = \frac{1}{z} - F_\mu \left( \frac{1}{z} \right), \quad
R^{\widetilde{\mu}, \mu}(z) = z \phi_{\widetilde{\mu}, \mu} \left( \frac{1}{z} \right).
\]

\subsection{Transformations}

For $\nu \in \Dist(d)$, the functional $\Phi[\nu]$ is determined by
\[
\eta^{\Phi[\nu]}(\mb{z}) = \sum_{i=1}^d z_i (1 + M^\nu(\mb{z})) z_i.
\]
See \cite{Belinschi-Nica-Free-BM} and \cite{AnsAppell3}.

In the algebraic setting, $\mf{B}$ is a bijection from $\Dist(d)$ to itself determined by
\[
R^{\mf{B}[\mu]} = \eta^{\mu}.
\]

Finally, for $\mu, \nu \in \Dist(d)$, the multivariate subordination distribution $\mu \boxright \nu \in \Dist(d)$ is defined via
\begin{equation}
\label{Eq:Multivariate-subord}
R^{\mu \boxright \nu} (\mb{z})
= R^{\mu} \Bigl( z_1 \left(1 + M^{\nu}(\mb{z}) \right), \ldots, z_d \left(1 + M^{\nu}(\mb{z}) \right) \Bigr) \left(1 + M^\nu(\mb{z}) \right)^{-1}.
\end{equation}
See Definition~1.1 in \cite{Nica-Subordination}.

\section{Multi-variate, algebraic results}
\label{Section:Algebraic}

The following proposition is the analog of the single-variable relation $G_{\mu \boxplus \nu}(z) = G_\nu(F_{\mu \boxright \nu}(z))$.

\begin{Prop}
\label{Prop:Multivariate-composition}
The subordination distribution $\mu \boxright \nu$ satisfies
\[
1 + M^{\mu \boxplus \nu} (\mb{z})
= \left(1 + M^{\mu \boxright \nu}(\mb{z}) \right) \left(1 + M^{\nu} \left(z_1 (1 + M^{\mu \boxright \nu}(\mb{z})), \ldots, z_d (1 + M^{\mu \boxright \nu}(\mb{z})) \right) \right).
\]
Consequently, for a fixed $\nu$, the map $\mu \mapsto \mu \boxright \nu$ is a bijection on $\Dist(d)$.
\end{Prop}

\begin{proof}
Note first that the equation
\[
1 + M^{\mu \boxplus \nu} (\mb{z})
= \Bigl(1 + M^{\lambda}(\mb{z}) \Bigr) \Bigl(1 + M^{\nu} \Bigl(z_1 (1 + M^{\lambda}(\mb{z})), \ldots, z_d (1 + M^{\lambda}(\mb{z})) \Bigr) \Bigr)
\]
has a unique solution $\lambda$. Indeed,
\[
(\mu \boxplus \nu)[x_{\vec{u}}] = \lambda[x_{\vec{u}}] + \nu[x_{\vec{u}}] + P_{\vec{u}}\left(\lambda[x_{\vec{v}}], \nu[x_{\vec{v}}] : \abs{\vec{v}} < \abs{\vec{u}} \right)
\]
for some polynomial $P_{\vec{u}}$.

Denote $w_i = z_i \left(1 + M^{\lambda}(\mb{z}) \right)$. Then
\[
\begin{split}
M^{\mu \boxplus \nu}(\mb{z})
& = R^{\mu \boxplus \nu} \Bigl( z_1 \left( 1 + M^{\mu \boxplus \nu}(\mb{z}) \right), \ldots, z_d \left( 1 + M^{\mu \boxplus \nu}(\mb{z}) \right) \Bigr) \\
& = R^{\mu \boxplus \nu} \Bigl( w_1 \left(1 + M^{\nu}(\mb{w}) \right), \ldots, w_d \left(1 + M^{\nu}(\mb{w}) \right) \Bigr) \\
& = R^{\mu} \Bigl( w_1 \left(1 + M^{\nu}(\mb{w}) \right), \ldots, w_d \left(1 + M^{\nu}(\mb{w}) \right) \Bigr) + M^\nu (\mb{w}).
\end{split}
\]
On the other hand,
\[
\begin{split}
M^{\mu \boxplus \nu} (\mb{z})
& = \left( 1 + M^\lambda(\mb{z}) \right) \left(1 + M^\nu(\mb{w}) \right) - 1 \\
& = M^\lambda(\mb{z}) \left(1 + M^\nu(\mb{w}) \right) + M^\nu(\mb{w}) \\
& = R^\lambda (\mb{w}) \left(1 + M^\nu(\mb{w}) \right) + M^\nu(\mb{w}).
\end{split}
\]
Combining these two equations,
\[
R^{\mu} \Bigl( w_1 \left(1 + M^{\nu}(\mb{w}) \right), \ldots, w_d \left(1 + M^{\nu}(\mb{w}) \right) \Bigr)
= R^\lambda (\mb{w}) \left(1 + M^\nu(\mb{w}) \right).
\]
Comparing with equation~\eqref{Eq:Multivariate-subord}, we see that $\lambda = \mu \boxright \nu$.

The equation in the proposition shows that given $\nu$ and $\lambda$, $\mu \boxplus \nu$, and consequently $\mu$, is uniquely determined. Conversely, the uniqueness statement above shows that given $\nu$ and $\mu$, $\mu \boxright \nu$ is uniquely determined.
\end{proof}

In the multi-variate, algebraic setting, all the results in Lemma~\ref{Lemma:Basic-properties} were proved in \cite{Nica-Subordination}, see Remark~1.2, Theorem~1.8, equation~(5.7), and Proposition~5.3. We will use them without proof.

\begin{Prop}
Let $\widetilde{\beta} \in \mf{R}^d$, $\widetilde{\gamma} > 0$, $\widetilde{\rho} \in \Dist(d)$, and $\set{\mu_t: t \geq 0} \subset \Dist(d)$ a free convolution semigroup. Define a two-state free convolution semigroup $\set{(\widetilde{\mu}_t, \mu_t) : t \geq 0}$ by
\[
R^{\widetilde{\mu}_t, \mu_t}(\mb{z}) = t \widetilde{\beta} \cdot \mb{z} + t \widetilde{\gamma} \sum_{i=1}^d z_i \left(1 + M^{\widetilde{\rho}}(\mb{z}) \right) z_i
\]
Define $\tau \in \Dist(d)$ via
\[
\mu = \mu_1 = \tau \boxright \widetilde{\rho}.
\]
Then
\[
\widetilde{\mu}_t = \delta_{t \widetilde{\beta}} \uplus \Phi[\widetilde{\rho} \boxplus \tau^{\boxplus t}]^{\uplus \widetilde{\gamma} t}.
\]
\end{Prop}

\begin{proof}
By the preceding proposition, for $\nu = \tau^{\boxplus t}$,
\[
1 + M^{\nu \boxplus \widetilde{\rho}} (\mb{z})
= \left(1 + M^{\nu \boxright \widetilde{\rho}}(\mb{z}) \right) \left(1 + M^{\widetilde{\rho}} \left( z_1 (1 + M^{\nu \boxright \widetilde{\rho}}(\mb{z})), \ldots, z_d (1 + M^{\nu \boxright \widetilde{\rho}}(\mb{z})) \right) \right).
\]
Since
\[
\mu_t = \tau^{\boxplus t} \boxright \widetilde{\rho}
= \nu \boxright \widetilde{\rho},
\]
this equation says that
\[
1 + M^{\widetilde{\rho} \boxplus \tau^{\boxplus t}} (\mb{z})
= \left(1 + M^{\mu_t}(\mb{z}) \right) \left(1 + M^{\widetilde{\rho}} \Bigl(z_1 \left(1 + M^{\mu_t}(\mb{z}) \right), \ldots, z_d \left(1 + M^{\mu_t}(\mb{z}) \Bigr) \right) \right).
\]
On the other hand,
\[
\begin{split}
\eta^{\widetilde{\mu}_t}(\mb{z})
& = R^{\widetilde{\mu}_t, \mu_t} \Bigl( z_1 \left(1 + M^{\mu_t}(\mb{z}) \right), \ldots, z_d \left(1 + M^{\mu_t}(\mb{z}) \right) \Bigr) \left(1 + M^{\mu_t}(\mb{z}) \right)^{-1} \\
& = t \widetilde{\beta} \cdot \mb{z} + t \widetilde{\gamma} \sum_{i=1}^d z_i \left(1 + M^{\mu_t}(\mb{z}) \right) \left(1 + M^{\widetilde{\rho}} \Bigl(z_1 \left(1 + M^{\mu_t}(\mb{z}) \right), \ldots, z_d \left(1 + M^{\mu_t}(\mb{z}) \right) \Bigr) \right) z_i.
\end{split}
\]
Combining these two equations, it follows that
\[
\eta^{\widetilde{\mu}_t}(\mb{z}) = t \widetilde{\beta} \cdot \mb{z} + t \widetilde{\gamma} \sum_{i=1}^d z_i \left(1 + M^{\widetilde{\rho} \boxplus \tau^{\boxplus t}} (\mb{z}) \right) z_i
\]
and
\[
\widetilde{\mu}_t = \delta_{t \widetilde{\beta}} \uplus \Phi[\widetilde{\rho} \boxplus \tau^{\boxplus t}]^{\uplus \widetilde{\gamma} t}. \qedhere
\]
\end{proof}

I am grateful to Hari Bercovici for a discussion leading to the following observations.

\begin{Cor}
\label{Cor:Compactly-supported}
Let $\set{(\widetilde{\mu}_t, \mu_t): t \geq 0}$ be a two-state free convolution semigroup of compactly supported probability measures such that $\widetilde{\mu}_1$ has non-zero variance. Then
\[
\mu = \tau \boxright \widetilde{\rho}
\]
and
\[
\mc{J}[\widetilde{\mu}_t] = \widetilde{\rho} \boxplus \tau^{\boxplus t}
\]
for some $\widetilde{\rho}$ a compactly supported probability measure, and $\tau$ a unital, not necessarily positive linear functional with non-negative variance, such that $\abs{\tau[x^n]} \leq C^n$ for some $C$
\end{Cor}

\begin{proof}
The result follows by applying the preceding proposition in the case $d=1$, when every compactly supported two-state free convolution semigroup is of the form specified in that proposition. Since for each $t$,
\[
\Var[\widetilde{\rho}] + t \Var[\tau] = \Var[\mc{J}[\widetilde{\mu}_t]] \geq 0,
\]
it follows that $\Var[\tau] \geq 0$. Positivity of $\widetilde{\rho}$ follows from positivity of $(\widetilde{\mu}_t, \mu_t)$, and the compact support of $\widetilde{\rho}$ and the growth conditions on $\tau$ from the compact support of  $(\widetilde{\mu}_t, \mu_t)$.
\end{proof}

\begin{Lemma}
\label{Lemma:Hari}
Let $\tau$ be a unital linear functional with positive variance such that $\abs{\tau[x^n]} \leq C^n$ for some $C$. Then $\tau^{\boxplus t}$ is positive definite, and so can be identified with a compactly supported measure, for sufficiently large $t$.
\end{Lemma}

\begin{proof}
Without loss of generality, we may assume that $\tau$ has mean $0$ and variance $1$. By assumption, the moments of $\tau$, and so also its free cumulants, grow no faster than exponentially. Therefore the $R$-transform of $\tau$ can be identified with an analytic function whose power series expansion at zero starts with $z$. It follows that for sufficiently large $t$, the $R$-transform of $D_{1/\sqrt{t}} \tau^{\boxplus t}$ satisfies the conditions of Theorem~2 of \cite{BerVoiSuperconvergence}. Applying that theorem, we conclude that $D_{1/\sqrt{t}} \tau^{\boxplus t}$, and so also $\tau^{\boxplus t}$, can be identified with a positive measure.
\end{proof}


\begin{thebibliography}{VDN92}

\bibitem[Ans09]{AnsAppell3}
Michael Anshelevich, \emph{Appell polynomials and their relatives. {III}.
  {C}onditionally free theory}, Illinois J. Math. \textbf{53} (2009), no.~1,
  39--66. \MR{MR2584934}

\bibitem[Ans10]{AnsEvolution}
\bysame, \emph{Free evolution on algebras with two states}, J. Reine Angew.
  Math. \textbf{638} (2010), 75--101. \MR{2595336}

\bibitem[Ans11]{Ans-Two-Brownian}
\bysame, \emph{Two-state free {B}rownian motions}, J. Funct. Anal. \textbf{260}
  (2011), no.~2, 541--565. \MR{2737412}

\bibitem[Ans13]{Ans-Generator}
\bysame, \emph{Generators of some non-commutative stochastic processes},
  Probab. Theory Related Fields \textbf{157} (2013), no.~3--4, 777--815.


\bibitem[AM12]{Ans-Mlot-Semigroups}
Michael Anshelevich and Wojciech M{\l}otkowski, \emph{Semigroups of
  distributions with linear {J}acobi parameters}, J. Theoret. Probab.
  \textbf{25} (2012), no.~4, 1173--1206. \MR{2993018}

\bibitem[BB04]{Belinschi-Bercovici-Partially-defined}
Serban~T. Belinschi and Hari Bercovici, \emph{Atoms and regularity for measures in a
  partially defined free convolution semigroup}, Math. Z. \textbf{248} (2004),
  no.~4, 665--674. \MR{2103535 (2006i:46095)}

\bibitem[BN08]{Belinschi-Nica-B_t}
Serban~T. Belinschi and Alexandru Nica, \emph{On a remarkable semigroup of
  homomorphisms with respect to free multiplicative convolution}, Indiana Univ.
  Math. J. \textbf{57} (2008), no.~4, 1679--1713. \MR{MR2440877 (2009f:46087)}

\bibitem[BN09]{Belinschi-Nica-Free-BM}
\bysame, \emph{Free {B}rownian motion and evolution towards
  {$\boxplus$}-infinite divisibility for {$k$}-tuples}, Internat. J. Math.
  \textbf{20} (2009), no.~3, 309--338. \MR{MR2500073}

\bibitem[BP99]{BerPatDomains}
Hari Bercovici and Vittorino Pata, \emph{Stable laws and domains of attraction
  in free probability theory}, Ann. of Math. (2) \textbf{149} (1999), no.~3,
  1023--1060, With an appendix by Philippe Biane. \MR{2000i:46061}

\bibitem[BV93]{BV93}
Hari Bercovici and Dan Voiculescu, \emph{Free convolution of measures with
  unbounded support}, Indiana Univ. Math. J. \textbf{42} (1993), no.~3,
  733--773. \MR{MR1254116 (95c:46109)}

\bibitem[BV95]{BerVoiSuperconvergence}
Hari Bercovici and Dan Voiculescu, \emph{Superconvergence to the central limit and
  failure of the {C}ram\'er theorem for free random variables}, Probab. Theory
  Related Fields \textbf{103} (1995), no.~2, 215--222. \MR{MR1355057
  (96k:46115)}

\bibitem[Kry07]{Krystek-Conditional}
Anna~Dorota Krystek, \emph{Infinite divisibility for the conditionally free
  convolution}, Infin. Dimens. Anal. Quantum Probab. Relat. Top. \textbf{10}
  (2007), no.~4, 499--522. \MR{MR2376439 (2009d:46118)}

\bibitem[Len07]{Lenczewski-Decompositions-convolution}
Romuald Lenczewski, \emph{Decompositions of the free additive convolution}, J.
  Funct. Anal. \textbf{246} (2007), no.~2, 330--365. \MR{MR2321046
  (2008d:28009)}

\bibitem[Maa92]{Maa92}
Hans Maassen, \emph{Addition of freely independent random variables}, J. Funct.
  Anal. \textbf{106} (1992), no.~2, 409--438. \MR{MR1165862 (94g:46069)}

\bibitem[M{\l}o10]{Mlotkowski-Fuss-Catalan}
Wojciech M{\l}otkowski, \emph{Fuss-{C}atalan numbers in noncommutative
  probability}, Doc. Math. \textbf{15} (2010), 939--955. \MR{2745687
  (2012c:46170)}

\bibitem[Nic09]{Nica-Subordination}
Alexandru Nica, \emph{Multi-variable subordination distributions for free
  additive convolution}, J. Funct. Anal. \textbf{257} (2009), no.~2, 428--463.
  \MR{2527024 (2010j:46121)}

\bibitem[NS96]{Nica-Speicher-Multiplication}
Alexandru Nica and Roland Speicher, \emph{On the multiplication of free
  {$N$}-tuples of noncommutative random variables}, Amer. J. Math. \textbf{118}
  (1996), no.~4, 799--837. \MR{98i:46069}

\bibitem[NS06]{Nica-Speicher-book}
\bysame, \emph{Lectures on the combinatorics of free probability}, London
  Mathematical Society Lecture Note Series, vol. 335, Cambridge University
  Press, Cambridge, 2006. \MR{MR2266879 (2008k:46198)}

\bibitem[VDN92]{VDN}
D.~V. Voiculescu, K.~J. Dykema, and A.~Nica, \emph{Free random variables}, CRM
  Monograph Series, vol.~1, American Mathematical Society, Providence, RI,
  1992, A noncommutative probability approach to free products with
  applications to random matrices, operator algebras and harmonic analysis on
  free groups. \MR{MR1217253 (94c:46133)}

\bibitem[Wan11]{Wang-Additive-c-free}
Jiun-Chau Wang, \emph{Limit theorems for additive conditionally free
  convolution}, Canad. J. Math. \textbf{63} (2011), no.~1, 222--240.
  \MR{2779139}

\end{thebibliography}

\def\cprime{$'$}
\providecommand{\bysame}{\leavevmode\hbox to3em{\hrulefill}\thinspace}
\providecommand{\MR}{\relax\ifhmode\unskip\space\fi MR }
\providecommand{\MRhref}[2]{%
  \href{http://www.ams.org/mathscinet-getitem?mr=#1}{#2}
}
\providecommand{\href}[2]{#2}

\end{document}